\documentclass[12pt,reqno]{amsart}
\usepackage{amsmath}
\usepackage{fullpage}
\theoremstyle{plain}
\newtheorem{thm}{Theorem}[section]

\theoremstyle{definition}
\newtheorem{defn}[thm]{Definition}
\newtheorem{ex}[thm]{Example}

\numberwithin{equation}{section}
\newcommand{\R}{\mathbb{R}}

\newcommand{\F}{\mathcal{F}}

\begin{document}
\title[A thermostat with delay]{Nontrivial solutions of a parameter-dependent heat flow problem with deviated arguments}
%\date{today}

% ----------------------------------------------------------------


\author[A. Calamai]{Alessandro Calamai}
\address{Alessandro Calamai, 
Dipartimento di Ingegneria Civile, Edile e Architettura,
Universit\`{a} Politecnica delle Marche
Via Brecce Bianche
I-60131 Ancona, Italy}%
\email{calamai@dipmat.univpm.it}%

\author[G. Infante]{Gennaro Infante}
\address{Gennaro Infante, Dipartimento di Matematica e Informatica, Universit\`{a} della
Calabria, 87036 Arcavacata di Rende, Cosenza, Italy}%
\email{gennaro.infante@unical.it}%

\begin{abstract}
By means of a recent Birkhoff-Kellogg type theorem set in affine cones, we discuss the solvability of a parameter-dependent thermostat problem subject to deviated arguments. We illustrate in a specific example the constants that occur in our theory.
\end{abstract}

\maketitle

\section{Introduction}

In this paper we investigate the existence of \emph{nontrivial} solutions of the 
second order parameter-dependent differential equation
\begin{equation} \label{2ord-intro}
u''(t)+\lambda f(t, u(t), u(\sigma (t))=0,\ t \in [0,1],
\end{equation}
with initial conditions
\begin{equation}\label{kmt-intro}
u(t)=\omega(t),\ t \in [-r,0],
\end{equation}
and the functional boundary condition (BC)
\begin{equation}\label{fbc-intro}
\beta u'(1)+u(\eta)=\lambda B[u],
\end{equation}
where $\beta>0$,  $\eta \in (0,1)$ and $0<\beta+\eta<1$, $\lambda$ is a parameter, $\sigma$, $\omega$, $f$ are suitable continuous functions and $B$ is a suitable functional. 

The motivation for studying this problem is that it arises in heat-flow problems. To illustrate this in a simple situation, let us consider the following special case of a problem with reflection of the argument:
\begin{equation} \label{ex-intro}
 \begin{cases}
u''(t)+\lambda f(t, u(t), u(-t))=0,\ & t \in [0,1], \\
u(t)=\omega(t),\  & t \in [-1,0], \\
\beta u'(1)+u(\frac14)=0.
 \end{cases}
\end{equation}
The BVP \eqref{ex-intro} describes the steady-states of temperature $u$ of a heated bar of length 2.  
Half of the bar is kept at a prescribed temperature $\omega$ and a controller, placed at the right end of the bar, is reacting to a sensor placed in the point $t=\frac14$. The presence of the term depending on the reflection of the argument illustrates the influence of the left hand side of the bar on the right hand side; a physical problem of this kind occurs in the case of a light bulb, see~\cite{ac-gi-at-tmna} for a detailed description. For simplicity here we have taken $B\equiv 0$ in the BCs; a nonzero term $B$ can be used to model complicated controller, possibly with non-linear response. There exists a wide literature on similar kinds of heat-flow problems, for brevity we mention the papers~\cite{acgi16, genupa, fama, df-gi-jp-prse, gijwnodea, gijwems,
 nipime, SZY18, jwpomona, jwwcna04, webb-therm} 
 in the case of \emph{linear} controllers and \cite{gi-caa, kamkee, kamkee2, kapala, nipime, palamides} for \emph{nonlinear} controllers.
Note that different choices of the function $\sigma$ leads to different situations; for example a spatial delay can be modelled by $\sigma(t)=t-\tau$, where $\tau\leq r$; this has been done, for example, in the context of delay equations in~\cite{acgi16}. More general deviated arguments occurring in heat-flow problems have been studied in \cite{ac-gi-at-tmna, rub-rod-lms}.

In order to discuss the solvability of the BVP~\eqref{2ord-intro}-\eqref{kmt-intro}-\eqref{fbc-intro}, we adopt a topological approach, based on a recent variant \cite{acgi2} in affine cones of the celabrated Birkhoff--Kellogg theorem. We provide an example where we illustrate the constants that occur in our theory. 

Topological approaches in affine cones have been exploited recently in \cite{acgi16, acgi2, acgi3, djeb2014}.  The setting of affine cones seems to be helpful when dealing with equations with delay effects. Here we prove the existence of nontrivial solutions $(u,\lambda)$ of the BVP \eqref{2ord-intro}-\eqref{kmt-intro}-\eqref{fbc-intro}, by means of an associated \emph{perturbed} Hammerstein integral equation set in a suitable traslate of a cone of functions that are allowed to \emph{change sign}. We mention that another variant of the Birkhoff--Kellogg theorem, due to Krasnosel'ski\u{i} and Lady\v{z}enski\u{\i} \cite{Kra-Lady}, has been used in \cite{giems} for a cone of sign changing functions with vertex in the origin under different sets of boundary conditions and without the presence of deviated arguments.

Our results are new and complement the previous literature. In particular we use a different topological tool w.r.t.~\cite{ac-gi-at-tmna, acgi16}, where the fixed point index is used directly. Furthermore here we study a BVP which is different from the one in \cite{acgi16}, due to the presence of the deviated arguments and the nonlinearity within the BCs, and from the one in~\cite{ac-gi-at-tmna}, due to the presence of the datum $\omega$. 

\section{Solutions of perturbed integral equations in affine cones}
We recall some useful notation.
Let $(X,\| \, \|)$ be a real Banach space. A \emph{cone} $K$ of $X$  is a closed set with $K+K\subset K$, $\mu K\subset K$ for all $\mu\ge 0$ and $K\cap(-K)=\{0\}$.
For $y\in X$, the \emph{translate} of the cone $K$ is defined as
$$
K_y:=y+K=\{y+x: x\in K\}.
$$
Given a bounded and open (in the relative
topology) subset $\Omega$  of $K_y$, we denote by $\overline{\Omega}$ and $\partial \Omega$
the closure and the boundary of $\Omega$ relative to $K_y$.
Given an open
bounded subset $D$ of $X$ we denote $D_{K_y}=D \cap K_y$, an open subset of
$K_y$.

We can now state a Birkhoff--Kellogg type result, which provides the existence of non-trivial solutions of parameter-dependent  functional equations in cone translates.
\begin{thm}[\cite{acgi2}, Corollary 2.4] \label{BK-transl-norm}
Let $(X,\| \, \|)$ be a real Banach space, $K\subset X$ be a cone and
 $D\subset X$ be an open bounded set with $y \in D_{K_y}$ and
$\overline{D}_{K_y}\ne K_y$. Assume that $\F:\overline{D}_{K_y}\to K$ is
a compact map and assume that 
$$
\inf_{x\in \partial D_{K_y}}\|\mathcal{F} (x)\|>0.
$$
Then there exist $x^*\in \partial D_{K_y}$ and $\lambda^*\in (0,+\infty)$ such that $x^*= y+\lambda^* \mathcal{F} (x^*)$.
\end{thm}

In order to apply Theorem~\ref{BK-transl-norm}, we make some assumptions on the following \emph{perturbed} Hammerstein integral equation:
\begin{equation}\label{eqhamm}
u(t)=\psi(t) + \lambda \Bigl( \int_{0}^{1} k(t,s)g(s) f(s, u(s), u(\sigma (s))\,ds +  \gamma(t) B[u] \Bigr)=:\psi(t) + \lambda \F u(t),\ t \in [-r, 1]
\end{equation}
where $B$ is a suitable (possibly nonlinear) functional in the space $C([-r, 1], \R)$, endowed of the usual supremum norm, $\|u\|_{[-r, 1]}$. More in general, given a compact interval $I\subset \R$, we denote by $C(I, \R)$ the Banach space of the
continuous functions defined on $I$ with the usual norm, $\|u\|_{I}$. We assume the following conditions: 
\begin{enumerate}
\item [$(C_{1})$] The function $\psi: [-r,1] \to \R$ is continuous.
\item [$(C_{2})$] The kernel $k:[-r,1] \times [0,1] \to\R$ is measurable,
verifies $k(t,s)=0$ for all $t\in[-r,0]$ and almost every (a.\,e.) $s \in[0,1]$, and
for every $\bar t \in
[0,1]$ we have
\begin{equation*}
\lim_{t \to \bar t} |k(t,s)-k(\bar t,s)|=0 \;\text{ for a.\,e. } s \in
[0,1].
\end{equation*}{}
\item [$(C_{3})$]
 There exist a subinterval $[a,b] \subseteq [0,1]$, a measurable function
$\Phi$ with $\Phi \geq 0$ a.\,e., and a constant $c_1=c_1(a,b) \in (0,1]$ such that
\begin{align*}
|k(t,s)|\leq \Phi(s) \text{ for  all }  &t \in [0,1] \text{ and a.\,e. } \, s\in [0,1],\\
k(t,s) \geq c_1\,\Phi(s) \text{ for  all } &t\in [a,b] \text{ and a.\,e. } \, s \in [0,1].
\end{align*}{}
\item [ $(C_{4})$] The function $g:[0,1] \to \R$ is measurable, $g(t) \geq 0$ a.\,e. $t \in [0,1]$,
and satisfies that $g\,\Phi \in L^1[0,1]$
and $\int_a^b \Phi(s)g(s)\,ds >0$.{}
\item [$(C_{5})$] $f: [0,1] \times \R \times \R \to [0,\infty)$ satisfies some 
Carath\'eodory-type conditions; namely, $ f(\cdot, u, v)$ is measurable for each fixed
$u$ and $v$ in $\R$, $f(t,\cdot,\cdot)$ is continuous for a.\,e. $t\in [0,1]$, and for each $R>0$, there exists $\phi_{R} \in
L^{\infty}[0,1]$ such that{}
\begin{equation*}
f(t,u,v)\le \phi_{R}(t) \;\text{ for all } \; (u,v)\in
[-R,R]\times [-R,R],\;\text{ and a.\,e. } \; t\in [0,1].
\end{equation*}{}
\item [ $(C_{6})$]  The function $\sigma:[0,1]\to [-r,1]$ is continuous.{}
\item [ $(C_{7})$]
The function $\gamma: [-r,1] \to \R$ is continuous, $\gamma\not\equiv 0$ and
such that $\gamma(t)=0$ for all $t\in[-r,0]$; moreover there exists $c_{2} \in(0,1] \;\text{such that}\;
\gamma(t) \geq c_{2}\|\gamma\|_{[0,1]} \;\text{for all }\; t \in [a,b]$.
\end{enumerate}
In the Banach space $C([-r, 1], \R)$ we define the cone
$$
K_0=\{u\in C([-r, 1], \R): u(t)=0\ \text{for all}\ t\in[-r,0],  \min_{t \in [a,b]}u(t)\geq c \|u\|_{[0,1]}\},
$$
where $c=\min\{c_1,c_2\}$. Note that $K_0 \neq \{0\}$ since $\gamma \in K_0$ and, furthermore, that the functions in $K_0$ are non-negative in the subset $[a,b]$ and may \emph{change sign} in $[0,1]$. The cone $K_0$ has been essentially introduced~\cite{acgi16}, as a modification of a cone introduced in~\cite{gijwjiea}. 

We consider the following translate of the cone $K_0$,
$$K_\psi=\psi + K_0 = \{\psi +u : u \in K_0\},$$ with the subset
$$K_{0,\rho}:=\{u\in K_0: \|u\|_{[0,1]} <\rho\}$$
and the corresponding translate
$$K_{\psi,\rho}:= \psi + K_{0,\rho}.$$
Note that
$\partial K_{\psi,\rho} = \psi + \partial K_{0,\rho}$
and that
$u\in  K_{\psi}$ means that $u=\psi+v$ with $v \in  K_{0}$ and, therefore, we have
\begin{equation*}%\label{norm-tra}
 \|u\|_{[-r,1]} =\max\{ \|\psi\|_{[-r,0]}, \|\psi+v\|_{[0,1]}\}.
\end{equation*}
We can now state our existence result.
\begin{thm}\label{eigen}Let $\rho \in (0,+\infty)$ and assume the following conditions hold. 

\begin{itemize}

\item[$(a)$] 
There exist $\underline{\delta}_{\rho} \in C([0,1],\R_+)$ such that
\begin{equation*}
f(t,u,v)\ge \underline{\delta}_{\rho}(t),\ \text{for every}\ t\in [a,b]\ \text{with}\ \max\{|u|,|v|\}\leq 
 \rho+\|\psi\|_{[-r,1]}.
\end{equation*}

\item[$(b)$] 
$B: \overline K_{\psi,\rho} \to \R_{+}$ is continuous and bounded, in particular let $\underline{\eta}_{\rho}\in [0,+\infty)$ be such that 
\begin{equation*}
B[u]\geq \underline{\eta}_{\rho},\ \text{for every}\ u\in \partial K_{\psi,\rho}.
\end{equation*}
\item[$(c)$] 
The inequality 
\begin{equation}\label{condc}
\sup_{t\in [a,b]}\Bigl\{ \gamma(t)\underline{\eta}_{\rho}+\int_{a}^{b}  k(t,s) \underline{\delta}_{\rho} (s)\,ds\Bigr\}>0 
\end{equation}
holds.
\end{itemize}
Then there exist $\lambda_\rho\in (0,+\infty)$ and $u_{\rho}\in \partial K_{\psi,\rho}$ that satisfy the integral equation~\eqref{eqhamm}.
\end{thm}

\begin{proof}
We firstly show that the operator operator $\F$ maps
$\overline{K}_{\psi,\rho}$ into $K_0$ and is compact. 
Take $u \in \overline{K}_{\psi,R}$; we have to prove that $\F u  \in K_0$.
First of all observe that our assumptions imply that $\F u$ is continuous on $[-r,1]$ and that
$\F u(t)=0$ for $t \in [-r,0]$.
Now, for every $t \in [0,1]$ we have
\begin{align*}
|\F u(t) | & \leq \int_{0}^{1}
|k(t,s)| g(s) f(s, u(s), u(\sigma (s)) \,ds + |\gamma(t)| B[u] \\
&\leq    \int_{0}^{1} \Phi(s)g(s) f(s, u(s), u(\sigma (s))\,ds +\|\gamma\|_{[0,1]} B[u],
\end{align*}
moreover, for $t\in [a,b]$,
\begin{equation}\label{inv}
\F u(t)  \geq c_1
\int_{0}^{1} \Phi(s)g(s) f(s, u(s), u(\sigma (s))\,ds
+ c_2 \|\gamma\|_{[0,1]}B[u]
\geq c\|\F u \|_{[0,1]}.
\end{equation}
Taking the minimum for $t \in [a,b]$ in \eqref{inv} yields $\F u \in K_0$, as desired.
The compactness of $\F$ follows in a similar way as in the proof of Theorem 3.2 of~\cite{acgi16}, since $B$ is continuous and bounded.

Now, take $u\in \partial K_{\psi, \rho}$. Then we have, for $t \in [a,b]$,
\begin{multline*}
 \F u(t)= \Bigl( \int_{0}^{1} k(t,s)g(s) f(s, u(s), u(\sigma (s))\,ds +  \gamma(t) B[u] \Bigr)\\ 
 \geq \Bigl( \int_{a}^{b} k(t,s)g(s) f(s, u(s), u(\sigma (s))\,ds +  \gamma(t) B[u] \Bigr) \\
 \geq \gamma(t)\underline{\eta}_{\rho}+\int_{a}^{b}  k(t,s) \underline{\delta}_{\rho} (s)\,ds.
\end{multline*}
Therefore we have
\begin{equation}\label{estbk}
\|  \F u\|_{[-r, 1]}\geq  \|  \F u\|_{[a, b]}  \geq \sup_{t\in [a,b]}\Bigl\{ \gamma(t)\underline{\eta}_{\rho}+\int_{a}^{b}  k(t,s) \underline{\delta}_{\rho} (s)\,ds\Bigr\}.
\end{equation}
Note that the RHS of \eqref{estbk} does not depend on the particular $u$ chosen. Hence,
$$
\inf_{u\in \partial K_{\psi, \rho}}\|  \F u\|_{[-r, 1]} \geq \sup_{t\in [a,b]}\Bigl\{ \gamma(t)\underline{\eta}_{\rho}+\int_{a}^{b}  k(t,s) \underline{\delta}_{\rho} (s)\,ds\Bigr\}>0,
$$
and the result follows by Theorem~\ref{BK-transl-norm}.
\end{proof}

\section{Nontrivial solutions of the BVP}
We turn our attention back to the BVP
\begin{equation} \label{Syst}
 \begin{cases}
u''(t)+\lambda f(t, u(t), u(\sigma (t))=0,\ & t \in [0,1], \\
u(t)=\omega(t),\  & t \in [-r,0], \\
\beta u'(1)+u(\eta)=\lambda B[u].
 \end{cases}
\end{equation}
In order to apply the previous theory to the BVP~\eqref{Syst}, we proceed by  means of a superposition principle as in Section 3 of~\cite{acgi3}, in the spirit of \cite{acgi16, gi-wil, gi-jw-ems-06}.

We begin by considering the BVP
\begin{equation*} %\label{Syst-om}
 \begin{cases}
u''(t)+y(t)=0,\ & t \in [0,1], \\
u(t)=0,\
\beta u'(1)+u(\eta)=0,
 \end{cases}
\end{equation*}
which has the 
has the unique solution
$$
u(t)= \int_{0}^{1} \hat{k}(t,s)y(s)ds,
$$
where the Green's function (see for example \cite{gi-jw-ems-06}) is given by
$$
\hat{k}(t,s)=\frac{\beta t}{\beta+\eta} +\dfrac{t}{\beta+\eta}(\eta-s)H(\eta-s)-
 (t-s) H(t-s),
$$
where
$$
H(\tau)=
\begin{cases}
1,\ & \tau \geq 0, \\0,\ & \tau<0,
\end{cases}
$$
thus we take 
\begin{equation}\label{def-ker}
k(t,s)=\hat{k}(t,s) H(t).
\end{equation}
With the choice of 
$[a,b] \subset (0, \beta+\eta) \subset (0, 1)$,
the hypotheses $(C_2)-(C_3)$ are satisfied (see \cite{acgi16}) when 
$$
\Phi (s)=\begin{cases}
s, & \text{ for }\beta+\eta\geq \frac{1}{2}, \\
\left[\dfrac{1 -(\beta + \eta )}{\beta+\eta}\right]s, & \text{ for
}\beta+\eta< \frac{1}{2},
\end{cases}
$$
and 
\begin{equation}\label{cbcA}
c_1=\begin{cases}
\min \Bigl\{ \dfrac{a \beta}{\beta+\eta},
\dfrac{\beta+\eta-b}{\beta+\eta} \Bigr\}, & \text{ for }\beta+\eta\geq \frac{1}{2}, \\
\min \Bigl\{ \dfrac{a \beta}{1 -(\beta + \eta)},
\dfrac{\beta+\eta-b}{1 -(\beta + \eta)} \Bigr\}
, & \text{ for
}\beta+\eta< \frac{1}{2}.
\end{cases}
\end{equation}
Note that also $(C_4)$ holds since $\int_a^b \Phi(s)\,ds >0$.
 
Now observe that the function $\hat{\gamma}(t)= \dfrac{ t}{\beta+\eta}$ satisfies the BVP
\begin{equation*}% \label{gambvp}
 \begin{cases}
u''(t)=0,\ t \in [0,1], \\
u(0)=0,\
\beta u'(1)+u(\eta)=1.
 \end{cases}
\end{equation*}
Thus we may take 
\begin{equation} \label{def-gam}
\gamma(t)= \hat{\gamma}(t)H(t),
\end{equation}
and note that $(C_7)$ is satisfied with $c_2=a$ (see \cite{acgi16}). Note also that $c_2\geq c_1$, hence we take $c=c_1$.

Finally note that $\hat{\varphi}(t)=\dfrac{(\beta + \eta )-t}{\beta+\eta}$ solves the BVP
\begin{equation*}% \label{phibvp}
 \begin{cases}
u''(t)=0,\ t \in [0,1], \\
u(0)=1,\
\beta u'(1)+u(\eta)=0
 \end{cases}
\end{equation*}
Thus, define
\begin{equation} \label{def-psi}
\psi(t) =
\begin{cases}
\omega(t),\ & t \leq 0, \\
\hat{\varphi}(t) \omega(0),\ & t > 0,
\end{cases}
\end{equation}
and observe that, by construction, $\psi$ is continuous on $[-r,1]$, since $\omega$ is assumed to be continuous on $[-r,0]$. 
Note that $\psi$, the vertex of the affine cone that we utilize, is built from the initial datum $\omega$, for which we do not require homogeneity in $0$, as in \cite{acgi3}.

Summing up, we can consider the integral equation
\begin{equation}\label{eqhamm-sol}
u(t)=\psi(t) + \lambda \Bigl( \int_{0}^{1} k(t,s) f(s, u(s), u(\sigma (s))\,ds +  \gamma(t) B[u] \Bigr),\quad t \in [-r, 1]
\end{equation}
where $k$ is as in \eqref{def-ker}, $\gamma$ is as in \eqref{def-gam} and $\psi$ is as in \eqref{def-psi}.
\begin{defn}
By a solution of  the BVP \eqref{Syst}
we mean a solution $u \in  C([-r, 1], \R)$
of the integral equation \eqref{eqhamm-sol}.
\end{defn}
With the above ingredients, we can state the following existence result.
\begin{thm} \label{example-thm}
Let $f: [0,1] \times \R \times \R \to [0,\infty)$ and $\sigma:[0,1]\to [-r,1]$ be continuous. Let $[a,b] \subset (0, \beta+\eta) \subset (0, 1)$,
and let $c=c_1$ as in \eqref{cbcA}. 
Let $\rho \in (0,+\infty)$ and assume that conditions 
$(a)$, $(b)$ and $(c)$ of Theorem \ref{eigen} hold. 
Then there exist $\lambda_\rho\in (0,+\infty)$ and $u_{\rho}\in \partial K_{\psi,\rho}$ that satisfy the BVP \eqref{Syst}.
\end{thm}

In the next illustrating example we consider a specific BVP of type
\eqref{Syst}, with the choice $r=1$ and $\sigma(s)=-s$, namely an equation with reflection of the argument.

\begin{ex}%\label{examp}
We  consider the  BVP

\begin{equation} \label{Syst-ex}
 \begin{cases}
u''(t)+\lambda  te^{u(t)+2u(-t)}=0,\ & t \in [0,1], \\
u(t)=\sqrt{1+t},\  & t \in [-1,0], \\
\frac14 u'(1)+u(\frac14)=\lambda \int_{-1}^1t^2(u(t))^2\, dt.
 \end{cases}
\end{equation}

Thus the function $\psi$ is given by
$$\psi(t) =
\begin{cases}
\sqrt{1+t},\ &  -1 \leq t \leq 0, \\
1-2t,\ & 0< t \leq 1.
\end{cases}
$$

Now choose $\rho\in (0,+\infty)$ and note that $\|\psi\|_{[-1,1]}=1$. We may take 
$$[a,b]=\left[\frac18,\frac14\right],\ 
\underline{\eta}_{\rho}(t)=0, \ \underline{\delta}_{\rho}(t)=te^{-3(1+\rho)}.$$

Therefore \eqref{condc} reads
$$
\sup_{t\in [\frac18,\frac14]}\Bigl\{e^{-3(1+\rho)} \int_{\frac18}^{\frac14}  k(t,s) s\,ds\Bigr\}\geq  e^{-3(1+\rho)}\int_{\frac18}^{\frac14}  \frac{1}{16} s^2 \,ds=\frac{7e^{-3(1+\rho)}}{24576} >0,
$$
which implies that \eqref{condc} is satisfied for every $\rho \in (0,+\infty)$.

Thus we can apply Theorem~\ref{example-thm} obtaining uncountably many pairs of solutions and parameters $(u_{\rho}, \lambda_{\rho})$ for the BVP
 \eqref{Syst-ex}.
\end{ex}

\section*{Acknowledgements}
The authors were partially supported by
 the Gruppo Nazionale per l'Analisi Matematica, la Probabilit\`a e le loro Applicazioni (GNAMPA) of the Istituto Nazionale di Alta Matematica (INdAM).
G.~Infante is a member of the UMI Group TAA  ``Approximation Theory and Applications''.  This paper was partially written during the
visit of A.~Calamai to the Dipartimento di Matematica e Informatica of the 
Universit\`a della Calabria. A.~Calamai is grateful to the people of the aforementioned
Dipartimento for their kind and warm hospitality.

\end{document}